\documentclass[a4paper]{article}
\usepackage{amsmath,amsthm,amssymb,enumerate,hyperref}
\usepackage[legacycolonsymbols]{mathtools}
\providecommand\given{}
\newcommand\SetSymbol[1][]{%
  \nonscript\:#1\vert
  \allowbreak
  \nonscript\:
  \mathopen{}}
\DeclarePairedDelimiterX\Set[1]{\{}{\}}{%
  \renewcommand\given{\SetSymbol[\delimsize]}
  #1}

\DeclarePairedDelimiterXPP\pospart[1]{}{(}{)}{^+}{#1}
\DeclarePairedDelimiterXPP\negpart[1]{}{(}{)}{^-}{#1}

\newcommand\R{\mathbb{R}}

\newcommand\N{\mathbb{N}}
\newcommand\PP{\mathcal{P}}

\renewcommand\AA{\mathcal{A}}

\newtheorem{thm}{Theorem}[section]
\newtheorem{prop}[thm]{Proposition}
\newtheorem{cor}[thm]{Corollary}
\theoremstyle{definition}

\newtheorem{ex}[thm]{Example}

\DeclareMathOperator*{\gr}{gr}

\DeclareMathOperator*{\conv}{conv}
\DeclareMathOperator*{\cone}{cone}
\DeclareMathOperator*{\dom}{dom}

\makeatletter
\newcommand{\leqnomode}{\tagsleft@true\let\veqno\@@leqno}
\newcommand{\reqnomode}{\tagsleft@false\let\veqno\@@eqno}
\makeatother

\title{Existence of solutions for polyhedral convex\\ set optimization problems}
\author{Andreas L{\"o}hne\thanks{Friedrich Schiller University Jena, Faculty of Mathematics and Computer Science, 07737 Jena, Germany, andreas.loehne@uni-jena.de}}

\begin{document}
\maketitle

\begin{abstract} \noindent Polyhedral convex set optimization problems are the simplest optimization problems with set-valued objective function. 
Their role in set optimization is  comparable to the role of linear programs in scalar optimization. Vector linear programs and multiple objective linear programs provide proper subclasses. 
In this article we choose a solution concept for arbitrary polyhedral convex set optimization problems out of several alternatives, show existence of solutions and characterize the existence of solutions in different ways. 
Two known results are obtained as particular cases, both with proofs being easier than the original ones:
The existence of solutions of bounded polyhedral convex set optimization problems and a characterization of the existence of solutions of vector linear programs.

\medskip

\end{abstract}

\section{Introduction}

\noindent A set-valued mapping $F:\R^n\rightrightarrows \R^q$ is called \emph{polyhedral convex} if its graph
$$ \gr F \coloneqq \Set{(x,y) \in \R^n\times \R^q \given y \in F(x)}$$
is a convex polyhedron. A \emph{polyhedral convex set-optimization problem} is the problem to minimize a polyhedral convex set-valued map $F$ with respect to the partial ordering $\supseteq$ (set inclusion).
Even though using the ordering $\supseteq$ is no restriction of generality, it is sometimes preferable to use set relations defined by a polyhedral convex ordering cone $C \subseteq \R^q$. Such a cone $C$ defines a reflexive and transitive ordering on $\R^q$ by the usual convention
$$ y^1 \leq_C y^2 \;\iff\; y^2-y^1 \in C.$$
This ordering can be lifted up to a reflexive and transitive ordering on the power set of $\R^q$, defined by
$$ A^1 \preccurlyeq_C A^2  \;\iff\; A^1 + C \supseteq A^2 + C,$$
see e.g.\ \cite[Section 2.2]{HamHeyLoeRudSch15} for bibliographic notes.
If $F$ is replaced by a set-valued mapping $F_C:\R^n\rightrightarrows \R^q$ defined by
$$ F_C(x) \coloneqq F(x) + C,$$
we obtain the equivalence
$$ F(x^1) \preccurlyeq_C F(x^2) \;\iff\; F_C(x^1) \supseteq F_C(x^2),$$
which means that the set inclusion $\supseteq$ is always sufficient to describe statements involving the ordering $\preccurlyeq_C$. The graph of $F_C$ is obtained directly from the graph of $F$ as
$$ \gr F_C = \gr F + (\Set{0} \times C).$$ 

Not only the ordering cone but also possible affine constraints can be subsumed into the objective function. Indeed, instead of minimizing the polyhedral convex set-valued mapping $F$ under the constraints $A x \geq b$, one can minimize the polyhedral convex mapping
$$ \bar F (x) \coloneqq \left\{\begin{array}{cl}
	 F(x) &\text{ if } Ax \geq b \\
	 \emptyset &\text{ otherwise. }
\end{array} \right.$$ 
This leads to the following problem formulation. 

Let $F:\R^n\rightrightarrows \R^q$ be a polyhedral convex set-valued map and let $C \subseteq \R^q$ be a polyhedral convex cone. Based on the ordering $\supseteq$ we consider the \emph{polyhedral convex set optimization problem}
\leqnomode
\begin{gather}\label{eq:p}\tag{P}
  \min F_C(x)\; \text{ s.t. }\; x \in \R^n.
\end{gather}
\reqnomode
We define in this article a solution concept for \eqref{eq:p}, which is similar to the one in \cite{HeyLoe}. It extends the solution concept for bounded polyhedral convex set optimization problems introduced in \cite{LoeSch13} in a natural way. We prove existence of solutions under reasonable assumptions. 

This article is based on the \emph{complete lattice approach} to set optimization (see e.g.\ \cite{HamHeyLoeRudSch15} and the references therein), where in contrast to other approaches in the literature (see e.g.\ \cite{KhaTamZal15} as well as the discussion in \cite[Sections 2.2, 3.2, 7.5]{HamHeyLoeRudSch15}) the solution concept is based on both \emph{infimum attainment} and \emph{minimality} (rather than on minimality only). In vector optimization, the complete lattice approach \cite{LoeTam07, Loe11} can be seen as a systematic continuation of ideas which were already proposed in 1987 by Dauer \cite{Dauer87} under the term ``objective space analysis''. In vector linear programming and polyhedral convex set optimization, the complete lattice approach is crucial for many applications and results. Here we mention only three examples: Based on this approach, the equivalence between polyhedral projection, vector linear programming and multiple objective programming is shown in \cite{LoeWei16}. A second example from mathematical finance is risk minimization for set-valued measures of risk \cite{JouMebTou04,HamHey10,HamHeyRud11,HamHeyLoeRudSch15}. A third example from multivariate statistics is the computation of Tukey depth regions by a reformulation into a vector linear program \cite{LoeWei}. 

\section{The bounded case}

The goal of this section is twofold. On the one hand side we prefer to introduce our main proof technique for the easier case of bounded problems to come to the point as quickly as possible. On the other hand side we obtain here an alternative and easier proof of the already known result \cite[Corollary 4.5]{LoeSch13} which states that every feasible and bounded polyhedral convex set optimization problem has a solution. While the existence result in \cite{LoeSch13} is obtained from the correctness proof of a method that computes solutions, we directly prove existence of solutions here.

The \emph{upper image} of \eqref{eq:p} is the set
$$ \PP\coloneqq C + \bigcup_{x \in \R^n} F(x).$$
$\PP$ is a convex polyhedron and we have $\bigcup_{x \in \R^n} (F(x) + C) = C + \bigcup_{x \in \R^n} F(x)$, see e.g.\ \cite{HeyLoe}. 
The upper image $\PP$ of \eqref{eq:p} can be interpreted as the infimum of $F$ over $\R^n$ with respect to $\preccurlyeq_C$ and likewise as the infimum of $F_C$ over $\R^n$ with respect to $\supseteq$, see e.g.\ \cite{HamHeyLoeRudSch15} for further details.

Problem \eqref{eq:p} is said to be \emph{bounded} if 
\begin{equation}\label{eq:bounded}
	\exists \ell \in \R^q: \; \forall x \in \R^n:\; \Set{\ell} \preccurlyeq_C F(x),
\end{equation}
or equivalently,
$$ \exists \ell \in \R^q: \;\Set{\ell}+C \supseteq \PP.$$

A finite subset $\bar S \subseteq \dom F\coloneqq \Set{x \in \R^n \given F(x) \neq \emptyset}$, $\bar S \neq \emptyset$ is called a \emph{finite infimizer} \cite{LoeSch13} of \eqref{eq:p} if
\begin{equation}\label{eq:infatt_b}
	\PP = C + \conv \bigcup_{x \in \bar S} F(x),
\end{equation}
where we denote by $\conv Y$ the convex hull of a set $Y \subseteq \R^q$.
Condition \eqref{eq:infatt_b} can be interpreted in the sense that the infimum $\PP$ of \eqref{eq:p} is attained at the finite set $\bar S$. A point $\bar x \in \R^n$ is called a \emph{minimizer} of \eqref{eq:p} if
\begin{equation}\label{eq_min}
	F(x) \preccurlyeq_C F(\bar x),\; x \in \R^n  \;\Rightarrow \;F(\bar x) \preccurlyeq_C F(x),
\end{equation}
or equivalently,
$$ \not\exists x \in \R^n:\quad F_C(x) \supsetneq F_C(\bar x).$$
A finite infimizer $\bar{S}$ of the bounded problem \eqref{eq:p} is called a {\em solution} to \eqref{eq:p} if all elements of  $\bar{S}$ are minimizers of \eqref{eq:p}.

This solution concept for bounded polyhedral convex set optimization problems was introduced in \cite{LoeSch13}. It extends the solution concept for vector linear programs introduced in \cite{Loe11} and is a ``polyhedral variant'' of the solution concepts for set optimization problems from \cite{HeyLoe11}. 

Since $F_C$ is a polyhedral convex mapping, an H-representation of $\gr F_C = \gr F + (\Set{0} \times C)$ exists. This means, there exist matrices $A\in \R^{m\times n}$, $B\in \R^{m\times q}$ and a vector $b \in \R^m$ such that
$$ (x,y) \in \gr F_C  \iff y \in F(x)+C \iff Ax + By \geq b.$$
Denoting by $B_i$ the $i$-th row of the matrix $B$ and setting $[m]\coloneqq \Set{1,\dots,m}$, we consider the linear program
\leqnomode
\begin{gather}\label{eq:lp}\tag{LP($\bar y$)}
	 \min \sum_{i=1}^m B_i y^i \quad \text{ s.t. } \quad  \left\{ \begin{array}{l}
		y^i \in F(x)+C,\; i \in [m] \\
		\bar y \in F(x)+C
	\end{array}\right.
\end{gather}
\reqnomode
with variables $(x,y^1,\dots,y^m) \in \R^n \times \R^q \times \dots \times \R^q$ and parameter $\bar y \in \R^q$.

\begin{prop}\label{prop1}
	\eqref{eq:lp} is feasible if and only if $\bar y \in \PP$.	
\end{prop} 
\begin{proof} 
	By the definition of $\PP$, we have $\bar y \in \PP$ if and only if there exists $\bar x \in \R^n$ such that $\bar y \in F(\bar x)+C$. Thus feasibility of  \eqref{eq:lp} implies $\bar y \in \PP$. Moreover, $\bar y \in \PP$ implies that $(\bar x,\bar y,\dots,\bar y)$ is feasible for \eqref{eq:lp}.
\end{proof}

\begin{prop} \label{prop2} If\eqref{eq:p} is bounded then \eqref{eq:lp} is bounded for every $\bar y \in \PP$.	
\end{prop} 
\begin{proof} 
	Let $\bar y \in \PP$. By the definition of $\PP$, there exists $\bar x \in \R^n$ such that $\bar y \in F(\bar x)+C$. For all $n \in \N$ and all $c \in C$ we have $\bar y + n c \in F(\bar x) + C$ and hence $B \bar y + n B c \geq b-A\bar x$. Thus $B c \geq 0$ for all $c \in C$. Let $\ell$ be a lower bound of \eqref{eq:p} according to \eqref{eq:bounded}. If $(x,y^1,\dots,y^m)$ is feasible for \eqref{eq:lp} then we have 
	$y^i \in F(x) + C \subseteq \Set{\ell} + C$ and hence $y^i - \ell \in C$ for all $i \in [m]$. Thus $B_i (y^i - \ell) \geq 0$ for all $i \in [m]$, which implies that \eqref{eq:lp} is bounded.
\end{proof}

\begin{prop}\label{prop3} Let $\bar y \in \PP$. If $(\bar x,\bar y^1,\dots,\bar y^m)$ is an optimal solution of $\eqref{eq:lp}$ then $\bar x$ is a minimizer of \eqref{eq:p}.	
\end{prop}
\begin{proof}
	Assume $\bar x$ is not a minimizer of \eqref{eq:p}, i.e., there exists $x \in \R^n$ such that $F(x)+C \supsetneq F(\bar x)+C$. Together with the feasibility of $(\bar x,\bar y^1,\dots,\bar y^m)$ for \eqref{eq:lp} this implies $\bar y^i \in F(x) + C$ for all $i \in [m]$ as well as $\bar y \in F(x) +C$. Moreover, there exists $y \in F(x)+C$ with $y \not\in F(\bar x)+C$. The latter condition implies that there is some $i \in [m]$ with $B_i y < b - A \bar x$. Without loss of generality let $i=1$. Then we have $B_1 y < b - A \bar x \leq B_1 y^1$. Then, the point $(x,y,\bar y^2,\dots,\bar y^m)$ is feasible for \eqref{eq:lp} and has a smaller objective value than the optimal solution $(x,\bar y^1,\bar y^2,\dots,\bar y^m)$, which is the desired contradiction.
\end{proof}

Proposition \ref{prop3} yields the following known result. 

\begin{cor}[{\cite[Corollary 4.5]{LoeSch13}}] If \eqref{eq:p} is feasible and bounded then \eqref{eq:p} has a solution.	
\end{cor}
\begin{proof}
	The upper image $\PP$ is a nonempty polyhedron with recession cone $C$. Thus it can be expressed by finitely many points $\bar y^i$ as
	$$ \PP = C +  \conv \Set{\bar y^1, \dots, \bar y^k}.$$
	Let $\bar y^i$ be a such a point. Then (LP($\bar y^i$)) is feasible by Proposition \ref{prop1} and bounded by Proposition \ref{prop2} and therefore it has an optimal solution whose $x$-component is denoted by $\bar x^i$. For all $i \in [k]$ we have $\bar y^i \in F(\bar x^i)+C$, which implies that the set $\bar S \coloneqq \Set{\bar x^1,\dots,\bar x^k}$ is an infimizer for \eqref{eq:p}. By Proposition \ref{prop3}, each $\bar x^i$ is a minimizer for \eqref{eq:p} and thus $\bar S$ is a solution to \eqref{eq:p}.
\end{proof} 

Note that in Proposition~\ref{prop3} we do not assume that \eqref{eq:p} is bounded but the existence of optimal solutions of \eqref{eq:lp} only. Conversely, if optimal solutions of \eqref{eq:lp} exist for all $\bar y \in \PP$, then \eqref{eq:p} is not necessarily bounded. Moreover, the converse implication in Proposition~\ref{prop2} does not hold.

\begin{ex} \label{ex1} Let $F:\R \rightrightarrows \R^2$ be given by its graph
	$$ \gr F \coloneqq \Set{(x,y_1,y_2) \in \R^3 \given y_1 \geq -x, \; y_2 \geq x,\; x \geq 0}$$
	and let $C \coloneqq \R^2_+$. Then for $x \in \dom F = \R_+$ we have 
	$$F(x) = F(x)+C = \left\{\begin{pmatrix}
		-x \\ \phantom{-}x
	\end{pmatrix}\right\}+C.$$
	Obviously, \eqref{eq:p} is not bounded. But \eqref{eq:lp} is of the form
	$$ \min y_1^1 + y_2^2 \text{ s.t. } y_1^1 \geq -x,\; y_2^1 \geq x,\; y_1^2 \geq -x,\; y_2^2 \geq x, \; \bar y_1 \geq -x,\; \bar y_2 \geq x, \; x \geq 0$$
	and thus for arbitrary feasible points $(x,y^1,y^2) \in \R \times \R^2\times \R^2$ the objective function is bounded below by $0$.	
\end{ex}

We close this section by noting that an H-representation of $\gr F_C$, which we used to show existence of solutions, is typically not known in practice. Usually, $\gr F_C$ is given by a P-representation, that is, by matrices $M_x \in \R^{m\times n}$, $M_y \in \R^{m\times q}$, $M_z \in \R^{m\times k}$ and a vector $c \in \R^m$ such that
$$ (x,y) \in \gr F_C \iff \exists z \in \R^k:\; M_x x + M_y y + M_z z \geq c.$$
An H-representation can be obtained from a P-representation by solving a polyhedral projection problem, see e.g.\ \cite{LoeWei16}. Equivalently, it can be obtained by solving a multiple objective linear program with $n+q+1$ objective functions \cite{LoeWei16}.
 However, both (mutually equivalent) problems are practicable only in case of $n+q$ being small. The method presented in this section is therefore not suitable for being used as a solution method, except for very small polyhedral convex set optimization problems. In particular, only a small number of variables is possible, which is unacceptably restricting.\ The approach in \cite{LoeSch13} is preferred as a solution method because it mainly operates in $\R^q$. Even though the results in \cite{LoeSch13} are formulated for H-representations of the involved graphs only, they can easily be extended to P-representations. For the computation of finite infimizers this was shown in \cite{HeyLoe} and for the computation of minimizers this is straightforward as the graph of the involved set-valued maps only appears in the feasible set of certain linear programs.

\section{The general case}

Polyhedral convex set optimization problems which are not necessarily bounded are investigated in \cite{HeyLoe}: Finite infimizers (sets of feasible points and directions where the infimum is attained) are defined and it is shown that they can be computed as in the bounded case by solving an associated vector linear program. Moreover, two variants of possible definitions for minimizers (points and directions which are minimal) are discussed. The authors of \cite{HeyLoe} point out examples of unbounded polyhedral convex set optimization problems where the condition
\begin{equation}\label{eq:exvlp}
	L(\PP) \cap (-C) \subseteq C
\end{equation}
is satisfied but solutions do not exist. Here $L(\PP)\coloneqq 0^+\PP \cap (-0^+\PP)$ denotes the {\em lineality space} of $\PP$. Furthermore, $0^+Y$ denotes the {\em recession cone} of a convex polyhedron $Y$, see e.g.\ \cite{Rockafellar70}. While  \eqref{eq:exvlp} implies the existence of solutions in case of vector linear programs \cite{Weissing20}, the condition turned out to be not sufficient for the existence of solutions of a polyhedral convex set optimization problem \cite{HeyLoe}. 

%%%%%%%%%%%%%%%%%%%%

In this article we propose a concept of minimality which is slightly different compared to \cite{HeyLoe} but very natural as we just use minimality with respect to the ordering cone $C$. In contrast to the solution concept presented below, in \cite[Definition 15]{HeyLoe} minimizing points are defined by the ordering cone $0^+\PP$, which is a superset of $C$. A minimizer as defined in \eqref{eq_min} with respect to an ordering cone $C$ is not necessarily a minimizer with respect to an ordering cone $D$ which is a subset or a superset of $C$. This phenomenon is illustrated by the following example. Note that the situation is different for a vector ordering. Here it is well-known that minimality with respect to a pointed convex cone $C$ implies minimality with respect to any pointed convex cone $D \subseteq C$.

\begin{ex}
	Let  $A_1 \coloneqq \Set{(0,0)^T}$, $A_2 \coloneqq \Set{(0,0)^T, (0,2)^T)}$, $A_3 \coloneqq \Set{(-1,1)}$ and consider the family of sets $\AA \coloneqq \Set{A_1,A_2,A_3}$. We define three ordering cones $C_1\coloneqq \cone \Set{(1,0)^T, (1,1)^T)}$, $C_2 \coloneqq \R^2_+$ and  $C_3\coloneqq \cone \Set{(1,0)^T, (-1,1)^T)}$. So we have $C_1 \subsetneq C_2 \subsetneq C_3$. We say that $A_i$ is $C_k$-minimal in $\AA$ if $A_j \preccurlyeq_{C_k} A_i$ implies $A_i \preccurlyeq_{C_k} A_j$. All three elements $A_1,A_2,A_3$ are $C_2$-minimal in $\AA$. But $A_3$ is not $C_3$-minimal in $\AA$ and $A_1$ is not $C_1$-minimal in $\AA$.
\end{ex} 

%%%%%%%%%%%%%%%%%%%

The modified solution concept we present below in this article was already propagated by Andreas H. Hamel and Frank Heyde in discussions about the results and examples of \cite{HeyLoe}. Here we follow their idea to accept that new phenomena can appear when vector linear programming is generalized to polyhedral convex set optimization. In addition we present a new motivation for this approach by providing an equivalent characterization of minimality and the associated solution concept. This characterization is unique in both the general polyhedral set-valued and the special vector-valued case.  

The \emph{recession mapping} of $F:\R^n\rightrightarrows \R^q$ is the set-valued mapping $0^+F:\R^n\rightrightarrows \R^q$ whose graph is the recession cone of the graph of $F$, i.e., we have
$$ \gr 0^+F = 0^+\gr F.$$
Using the convention
$$ y \in F_C(x) \iff y \in F(x)+C \iff Ax + B y \geq b,$$
as introduced in the previous section, we obtain
$$ y \in (0^+F)(x) + C \iff y \in (0^+F_C)(x) \iff Ax + By \geq 0.$$
The first equivalence follows from the fact 
$0^+(\gr F + (\Set{0} \times C)) = 0^+ \gr F + (\Set{0} \times C)$, which is true because the sets being involved are polyhedral convex, see e.g.\ \cite{Rockafellar70}. As a  direct consequence we also obtain $\dom (0^+F) = \dom (0^+F_C) = 0^+(\dom F)$.

We introduce the homogeneous version of \eqref{eq:p} as
\leqnomode
\begin{gather}\label{eq:phom}\tag{P$_{\text{hom}}$}
  \min (0^+F_C)(x)\; \text{ s.t. }\; x \in \R^n
\end{gather}
\reqnomode
and the homogeneous version of \eqref{eq:lp} as
\leqnomode
\begin{gather}\label{eq:lphom}\tag{LP$_{\text{hom}}$}
	 \min \sum_{i=1}^m B_i y^i \quad \text{ s.t. } \quad  \left\{ \begin{array}{l}
		y^i \in (0^+F)(x)+C,\; i \in [m] \\
		0 \in (0^+F)(x)+C.
	\end{array}\right.
\end{gather}
\reqnomode
Note that we obtain \eqref{eq:lphom} from \eqref{eq:lp} by setting all right-hand side entries of affine inequalities to zero.

A tuple $(\bar{S},\hat{S}$) of finite sets $\bar{S} \subseteq \dom F$, $\bar{S} \neq \emptyset$ and $\hat{S} \subseteq \dom (0^+F) \setminus \{0\}$ is called a {\em finite infimizer} \cite{HeyLoe} for problem \eqref{eq:p} if 
\begin{equation}{\label{finiteinfimizer}}
	\PP = C +  \conv \bigcup_{x \in \bar{S}} F(x) + \cone \bigcup_{x \in \hat{S}} (0^+F)(x).
\end{equation}
Here $\cone Y$ denotes the \emph{conic hull}, i.e., the smallest (with respect to set inclusion $\subseteq$) convex cone containing the set $Y$. In particular, we set $\cone \emptyset = \Set{0}$, which ensures that \eqref{finiteinfimizer} reduces to \eqref{eq:infatt_b} if $\hat S = \emptyset$.
 
An element $\bar{x} \in \dom F$ is called a {\em minimizer} or {\em minimizing point} for \eqref{eq:p} if 
\begin{equation*}
		F(x) \preccurlyeq_{C} F(\bar x),\; x \in \R^n \; \Rightarrow \; F(\bar x) \preccurlyeq_C F(x),
\end{equation*}
or equivalently
$$ \not\exists x \in \R^n:\; F_C(x) \supsetneq F_C(\bar x).$$
A nonzero element $\hat{x} \in \dom (0^+F)$ is called a {\em minimizing direction} for \eqref{eq:p} if
\begin{align*}
		(0^+F)(x) \preccurlyeq_C (0^+F)(\hat{x}),\; x \in \R^n \; \Rightarrow \; (0^+F)(\hat x) \preccurlyeq_C (0^+F)(x),
\end{align*}
or equivalently
$$ \not\exists x \in \R^n:\; (0^+F_C)(x) \supsetneq (0^+F_C)(\hat x).$$
A finite infimizer $(\bar{S},\hat{S})$ is called a {\em solution} to \eqref{eq:p} if all elements of  $\bar{S}$ are minimizing points and all elements of $\hat{S}$ (if any) are minimizing directions.

\begin{prop}\label{prop4} 
	Let \eqref{eq:p} be feasible.
	The following is equivalent:
	\begin{enumerate}[(i)]
		\item There exists $\bar y \in \PP$ such that \eqref{eq:p} has a minimizer $\bar x$ with $\bar y \in F(\bar x)+C$.
		\item For all $\bar y \in \PP$, \eqref{eq:p} has a minimizer $\bar x$ with $\bar y \in F(\bar x)+C$.
		\item There exists $\bar y \in \PP$ such that \eqref{eq:lp} has an optimal solution.
		\item For all $\bar y \in \PP$, \eqref{eq:lp} has an optimal solution.
		\item \eqref{eq:lphom} has an optimal solution.
		\item $0$ is a minimizer of \eqref{eq:phom}.
	\end{enumerate}
\end{prop}
\begin{proof} 
(vi) $\Rightarrow$ (v): Assume that \eqref{eq:lphom} has no optimal solution. Since $0 \in \R^{n+mq}$ is feasible, \eqref{eq:lphom} must be unbounded. Thus \eqref{eq:lphom} has a feasible point $(x,y^1,\dots,y^m)$ with negative objective value. Feasibility includes the condition $0 \geq - Ax$. Therefore, any $y \in (0^+F)(0)+C$ satisfies $By \geq 0 \geq -Ax$. This can be written as $y \in (0^+F)(x)+C$. Thus we have the inclusion $(0^+ F)(x) + C \supseteq (0^+ F)(0) + C$.   
The negative objective value implies the existence of $i \in [m]$ with $B_i y^i < 0$. Thus $B y^i \not\geq 0$, which can be expressed as $y^i \not\in (0^+F)(0)+C$. Since $y^i \in (0^+F)(x)+C$, we obtain $(0^+ F)(x) + C \supsetneq (0^+ F)(0) + C$ which means that
$0$ is not a minimizer of \eqref{eq:phom}. 		
	
(v) $\Rightarrow$ (iv): By Proposition \ref{prop1}, \eqref{eq:lp} is feasible for every $\bar y \in \PP$. The dual problem
of \eqref{eq:lphom} is feasible, by (v) and linear programming duality. It has the same feasible set as the dual program of \eqref{eq:lp}. In particular, the feasible set of the dual programs does not depend on $\bar y \in \PP$. For every $\bar y \in \PP$, both \eqref{eq:lp} as well as its dual program is feasible. Linear programming duality yields that both have optimal solutions.

[(iv) $\Rightarrow$ (iii)] and [(ii) $\Rightarrow$ (i)]:  Both statements are obvious since \eqref{eq:p} is feasible by assumption and therefore $\PP \neq \emptyset$.

[(iv) $\Rightarrow$ (ii)] and [(iii) $\Rightarrow$ (i)]: Both statements follow from Proposition \ref{prop3}.

(i) $\Rightarrow$ (vi): Assume that (vi) is violated, i.e., there exists $x \in \R^n$ such that $(0^+F)(x)+C \supsetneq (0^+F)(0)+C$. Since $0 \in (0^+F)(0)+C$ we have $0 \in (0^+F)(x)+C$ and thus $0 \geq - Ax$. Let $y \in F(\bar x)+C$, i.e., $B y \geq b - A\bar x$. For all $n \in \N$, we obtain $B y \geq b - A(\bar x + n x)$, i.e., $y \in F(\bar x + n x)+C$. Thus, for all $n \in \N$, the inclusion $F(\bar x + nx)+C \supseteq F(\bar x)+C$ holds. Taking some $y \in (0^+F)(x)+C$ such that $y \not\in (0^+F)(0)+C$, we have $B y \geq -A x$ and $By \not\geq 0$. For $\bar y$ and $\bar x$ from (i), the first inequality implies that $\bar y + n y \in F(\bar x + n x) + C$ for all $n \in \N$. The statement $By \not\geq 0$ yields that $\bar y + n y \not\in F(\bar x) + C$ for sufficiently large $n$. For this $n$ we have $F(\bar x + nx)+C \supsetneq F(\bar x)+C$ which means that $\bar x$ is not a minimizer of \eqref{eq:p}.
\end{proof}

\begin{thm}\label{thm1} There exists a solution for \eqref{eq:p} if and only if \eqref{eq:p} is feasible and one of the six equivalent conditions (i)-(vi) in Proposition \ref{prop4} is satisfied.	
\end{thm}
\begin{proof}
	``$\Rightarrow$'': Let $(\bar S, \hat S)$ be a solution to \eqref{eq:p}. Then $\bar S \neq \emptyset$, in particular, \eqref{eq:p} is feasible. Let $\bar x \in \bar S \subseteq \dom F$ and $\bar y \in F(\bar x) + C$. Then $\bar x$ is a minimizing point of \eqref{eq:p} and altogether we obtain that condition (i) of Proposition \ref{prop4} is satisfied for $\bar y$ and $\bar x$. 
	
	``$\Leftarrow$'': Feasibility of \eqref{eq:p} implies $\PP \neq \emptyset$. Thus $\PP$ has a V-representation
	$$ \PP = \conv \{\bar y^1,\dots,\bar y^s\} + \cone \{\hat y^1,\dots,\hat y^t\},$$
	where $s \geq 1$ and $t \geq 0$. By Proposition \ref{prop4} (ii), for every $i \in [s]$ there is a minimizer $\bar x^i$ for \eqref{eq:p} such that $\bar y^i \in F(\bar x^i) + C$. Moreover, by Proposition \ref{prop4} applied to \eqref{eq:phom} instead of \eqref{eq:p} (note that the homogeneous problem of \eqref{eq:phom} is \eqref{eq:phom}), for every $j \in [t]$ there is a minimizer $\hat x^i$ for \eqref{eq:phom} (which is a minimizing direction for \eqref{eq:p}) such that $\hat y^i \in (0^+F)(\hat x^i) + C$. It follows that $(\bar S, \hat S)$ for $\bar S \coloneqq \Set{\bar x^i \given i \in [s]}$ and $\hat S \coloneqq \Set{\hat x^j \given j \in [t],\, \hat x^j \neq 0}$ provides a solution to \eqref{eq:p}.
\end{proof}

We close this article by considering the particular case of a vector linear program
	\leqnomode
	\begin{gather}\label{eq:vlp}\tag{VLP}
	  \min\text{$_C$} Mx\; \text{ s.t. }\; A x \geq b,
	\end{gather}
	\reqnomode
where $M \in \R^{q\times n}$, $A \in \R^{m \times n}$, $b \in \R^m$ and $C \subseteq \R^q$ being a polyhedral convex ordering cone. A set-valued mapping $F$ is defined as
$$F(x)\coloneqq \left\{ \begin{array}{cl}
		\Set{Mx} &\text{  if } Ax \geq b \\
		\emptyset  &\text{ otherwise. }
	\end{array}\right. $$
With this $F$ and $C$, a tuple $(\bar S,\hat S)$ is a solution to \eqref{eq:p} if and only if it is a solution to \eqref{eq:vlp}, compare \cite{LoeWei16, Weissing20} for more details. 

We recover the following known result.	

\begin{cor}[{\cite[Theorem 4.1]{Weissing20}}]\label{cor2}
	A solution to \eqref{eq:vlp} exists if and only if \eqref{eq:vlp} is feasible and condition \eqref{eq:exvlp} is satisfied.
\end{cor}
\begin{proof} By a straightforward argumentation, we see that 
	$$ C + \Set{y} \supsetneq C \iff y \in - C \setminus C$$
	holds for $y \in \R^q$, where we only need that $C \subseteq \R^q$ is a convex cone here.
	Since $C \subseteq 0^+\PP$, \eqref{eq:exvlp} is equivalent to 
	\begin{equation}\label{eq:exvlp1}
		0^+\PP \cap (-C) \subseteq C.
	\end{equation}
	%Indeed, we have
	%$$ -0^+\PP \cap C = -0^+\PP \cap C \cap C = -0^+\PP \cap 0^+\PP \cap C.$$
	Assuming that $0$ is not a minimizer of \eqref{eq:phom}, we find some $x \in \R^n$ such that
	$(0^+F_C)(x) \supsetneq (0^+ F_C)(0)$. By the special form of $F$, this means there exists $x \in \R^n$ with $A x \geq 0$ such that $\Set{Mx}+C \supsetneq C$. This implies $Mx \in -C\setminus C$ and $Mx \in 0^+ \PP$. Thus \eqref{eq:exvlp1} is violated. 
	
	To prove the converse implication, let \eqref{eq:exvlp1} be violated, i.e., there exists some $y \in 0^+\PP \cap (-C \setminus C)$.
	Since $0^+\PP$ is the upper image of \eqref{eq:phom} (which is a consequence of e.g. \cite[Proposition 3]{HeyLoe}), there exists some $x \in \R^n$ with $Ax \geq 0$ such that $y \in \Set{Mx} + C$, whence $\Set{y} + C \subseteq \Set{Mx} + C$. From $y \in -C\setminus C$, we conclude $C + \Set{y} \supsetneq C$. Together we obtain 
 $(0^+ F_C)(x) = \Set{Mx} + C \supsetneq C = (0^+ F_C)(0)$, i.e., $0$ is not a minimizer of \eqref{eq:phom}.	
	
	The statement now follows from Theorem \ref{thm1}.
\end{proof}

The following example from \cite{HeyLoe} shows that the statement of Corollary \ref{cor2} for vector linear programs is not true for arbitrary polyhedral convex set optimization problems.
\begin{ex}\label{ex5}
	Let $C=\R^2_+$ and let $F: \R \rightrightarrows \R^2$ be defined by 
		$$ \gr F = \cone \left\{ \begin{pmatrix}1 \\0 \\ 0 \end{pmatrix},\; \begin{pmatrix} \phantom{-}1 \\ \phantom{-}2 \\-1 \end{pmatrix} \right\}.$$
		Since $\gr F$ is a cone, we have $F=0^+F$. Thus \eqref{eq:p} coincides with \eqref{eq:phom}. For $x \in \dom F = \R_+$, it can be easily verified that
		$$ (0^+F)(x) = F(x) = \conv\left\{ \begin{pmatrix} 0 \\ 0 \end{pmatrix},\; \begin{pmatrix} 2 x \\-x \end{pmatrix} \right\}.$$
		This implies $(0^+F_C)(1) \supsetneq (0^+F_C)(0)$. Thus $0$ is not a minimizer of \eqref{eq:phom}. By Theorem \ref{thm1}, the problem does not have a solution.
		
		However, the upper image is 
		$$\PP = 0^+\PP = \Set{y \in \R^2 \given y_1 + 2 y_2 \geq 0,\; y_1\geq 0}$$
		and we see that condition \eqref{eq:exvlp} is satisfied.
\end{ex}
  
%\subsection*{Acknowledgements}
%This research was motivated by
  
%\bibliography{ref}

\end{document}